\documentclass{amsart}

\usepackage{graphicx}
\usepackage{epsf}
\usepackage{epsfig}
\usepackage{psfrag}
\usepackage{amsmath,amssymb,enumerate}
\usepackage{subfigure}
\usepackage{color}
\usepackage{multirow}
\newtheorem{theorem}{Theorem}[section]
\newtheorem{lemma}[theorem]{Lemma}

\theoremstyle{remark}

\numberwithin{equation}{section}
\renewcommand{\thetable}{\thesection.\arabic{table}}
\renewcommand{\thefigure}{\thesection.\arabic{figure}}

\theoremstyle{definition}
\newtheorem{definition}[theorem]{Definition}
\newtheorem{example}[theorem]{Example}

\usepackage{bm}

\newcommand{\pa}{\partial}
\newcommand{\la}{\langle}
\newcommand{\ra}{\rangle}

\newcommand{\al}{\alpha}

\newcommand{\na}{\nabla}

\newcommand{\vp}{\varphi}

\newcommand{\Ga}{\Gamma}
\newcommand{\Om}{\Omega}
\newcommand{\de}{\delta}

\newcommand{\De}{\Delta}

\newcommand{\lam}{\lambda}

\newcommand{\R}{{\mathbb{R}}}

\newcommand{\T}{{\mathbb{T}}}
\newcommand{\E}{{\mathbb{E}}}
\newcommand{\cM}{\mathcal{M}}

\newcommand{\vep}{\varepsilon}
\newlength\savewidth

\newcommand{\be}{\begin{eqnarray}}
\newcommand{\ee}{\end{eqnarray}}
\newcommand{\nn}{\nonumber}
\newcommand{\ben}{\begin{eqnarray*}}
\newcommand{\een}{\end{eqnarray*}}

\newcommand{\cE}{{\mathcal E}}

\begin{document}
\title [A FEM for Elliptic Problems with Observational Boundary Data]
{A Finite Element Method for Elliptic Problems with Observational Boundary Data}

% author one information
\author{Zhiming Chen}
\address{LSEC, Institute of Computational Mathematics, Academy of
Mathematics and System Sciences, Chinese Academy of Sciences, Beijing 100190, China
and School of Mathematical Sciences, University of
Chinese Academy of Sciences, Beijing 100049, China}
%\curraddr{}
\email{zmchen@lsec.cc.ac.cn}
\thanks{This work was supported in part by the National Center for Mathematics and Interdisciplinary 
Sciences, CAS and China NSF under the grant 11021101 and 11321061.}
% author two information
\author{Rui Tuo}
\address{Institute of Systems Science,
Academy of Mathematics and System Sciences, Chinese Academy of Sciences,
Beijing 100190, China}
\email{tuorui@amss.ac.cn}
% author three information
\author{Wenlong Zhang}
\address{School of Mathematical Science, University of
Chinese Academy of Sciences, Beijing 100190, China.}
\email{zhangwl@lsec.cc.ac.cn}
\subjclass[2000]{65N30, 65N15}
% Use this one if you are using an older version of amsart.
%\subjclass{}
%\date{}

% at present the "communicated by" line appears only in ERA and PROC
%\commby{}

%\dedicatory{}
\date{\today}
\begin{abstract} In this paper we propose a finite element method for solving elliptic equations with the observational Dirichlet boundary data which may subject to random noises. The method is based on the weak formulation of Lagrangian multiplier. 
We show the convergence of the random 
finite element error in expectation and, when the noise is sub-Gaussian, in the Orlicz $\psi_2$-norm which implies
the probability that the finite element error estimates are violated decays exponentially. Numerical examples are included.
\end{abstract}

\maketitle

%%%%%%%%%%%%%%%%%%%%%%%%%%%%%%%%%%%%%%%%%%%%%%%%%%%%%%%%%%%%%%%
\section{Introduction}
\setcounter{equation}{0}
%%%%%%%%%%%%%%%%%%%%%%%%%%%%%%%%%%%%%%%%%%%%%%%%%%%%%%%%%%%%%%%

In many scientific and engineering applications involving partial differential equations, the input data such as sources or boundary conditions are usually given through the measurements which may subject to random noises. Let $\Om\subset \R^2$ be a bounded domain with smooth boundary $\Ga$. In this paper we consider the problem to find $u\in H^1(\Om)$ such that
\be\label{p1}
-\Delta u=f\ \ \mbox{in }\Om,\ \ \ \ u=g_0\ \ \mbox{on }\Ga.
\ee
Here $f\in L^2(\Om)$ is given but the boundary condition $g_0\in H^2(\Ga)$ is generally unknown. We assume we know the measurements 
$g_i=g_0(x_i)+e_i$, $i=1,2,\cdots,n$, where $\T=\{x_i: 1\le i\le n\}$ is the set of the measurement locations on the boundary $\Ga$ and $e_i$, $i=1,2,\cdots,n$, are independent identically distributed random variables over some probability space $(\mathfrak{X}, \mathcal{F},\mathbb{P})$ satisfying $\E[e_i]=0$ and $\E[e_i^2]=\sigma>0$. In this paper $\mathbb{P}$ denotes the probability measure and $\E[X]$ denotes the expectation of the random variable $X$. We remark that for simplicity we only consider the problem of observational Dirichlet boundary data in this paper and the problem with observational sources $f$ or
other type of boundary conditions can be studied by the same method.

A different perspective of solving partial differential equations with uncertain input data due to incomplete knowledge or inherent variability in the system has drawn considerable interests in recent years (see e.g. \cite{btz04, cd15, gwz14, tt15} and the references therein). The goal of those studies is to learn about the uncertainties in system outputs of interest, given information about the uncertainties in the system inputs which are modeled as random field. This goal usually leads to the mathematical problem of breaking the curse of dimensionality for solving partial differential equations having large number of parameters. 

The classical problem to find a smooth function from the knowledge of its observation at scattered locations subject to random noises is well studied in the literature \cite{w90}. One popular model to tackle this classical problem is to use the thin plate spline model \cite{d77, u88} which can be efficiently solved by using finite element methods \cite{am97, rha03, ctz16}. The scattered data in our problem \eqref{p1} are defined on the boundary of the domain and a straightforward application of the method developed in \cite{d77, u88, am97, rha03, ctz16} would lead to solve a fourth order elliptic equation on the boundary which would be much more expansive than the method proposed in this paper.

Our method is based on the following weak formulation of Lagrangian multiplier for \eqref{p1} in \cite{b74}:
Find $(u,\lam)\in H^1(\Om)\times H^{-1/2}(\Ga)$ such that
\be
(\na u,\na v)+\la\lam,v\ra&=&(f,v),\ \ \forall v\in H^1(\Om),\label{p2}\\
\la \mu,u\ra&=&\la \mu,g\ra,\ \ \forall\mu\in H^{1/2}(\Ga),\label{p3}
\ee
where $(\cdot,\cdot)$ is the duality pairing between $H^1(\Om)$ and $H^1(\Om)'$ which is an extension of the inner product of $L^2(\Om)$ and $\la\cdot,\cdot\ra$ is the duality pairing between $H^{1/2}(\Ga)$ and $H^{-1/2}(\Ga)$ which is an extension of the inner product of $L^2(\Ga)$. Let $\Om_h$ be a polygonal domain which approximates the domain $\Om$. Let $V_h\subset H^1(\Om_h)$ and $Q_h\subset L^2(\Ga)$ be the finite element spaces for approximating the field variable and the Lagrangian multiplier. Our finite element method is defined as follows: Find $(u_h,\lam_h)\in V_h\times Q_h$ such that 
\ben
(\na u_h,\na v_h)_{\Om_h}+\la\lam_h,v_h\ra_n&=&(I_hf,v_h)_{\Om_h},\ \ \forall v_h\in V_h,\\
\la \mu_h,u_h\ra_n&=&\la \mu_h,g\ra_n,\ \ \forall\mu_h\in Q_h,
\een
where $(\cdot,\cdot)_{\Om_h}$ is the inner product of $L^2(\Om_h)$, $\la \cdot,\cdot\ra_n$ is some quadrature rule for approximating $\la \cdot,\cdot\ra$, and $I_h$ is some finite element interpolation operator (we refer to section 2 for the precise
definitions). We remark that while the method of Lagrangian multiplier is one of the standard ways in enforcing Dirichlet boundary condition on smooth domains, it is essential here for solving the problem with Dirichlet observational boundary data even when the domain $\Om$ is a polygon. One can also combine the techniques developed in this paper with other weak formulations in \cite{s95} to deal with the observational Dirichlet boundary condition.

Our analysis in section 3 shows that
\be\label{p6}
%& &\E[\|u-u_h\circ\Phi_h^{-1}\|_{H^1(\Om)}+h^{1/2}\|\lam-\lam_h\|_{H^{-1/2}(\Ga)}]\\
%&\le&Ch|\ln h|^{1/2}(\|u\|_{H^2(\Om)}
%+\|f\|_{H^2(\Om)})+Ch^{-1}(\sigma n^{-1/2}),\label{p5}\\
\E\left[\|u-u_h\circ\Phi_h^{-1}\|_{L^2(\Om)}\right]\le Ch^2|\ln h|\De(u,f,g_0)+C|\ln h|(\sigma n^{-1/2}),
\ee
where $\De(u,f,g_0)=\|u\|_{H^2(\Om)}+\|f\|_{H^2(\Om)}+\|g_0\|_{H^2(\Ga)}$ and $\Phi_h:\Om_h\to\Om$ is the Lenoir homeomorphism defined in section 3. This error estimate suggests that in order to achieve the optimal convergence, one should take the number of sampling points satisfying $\sigma n^{-1/2}\le Ch^2$ to compute the solution over a finite element mesh of the mesh size $h$. For problems having Neumann or Robin boundary conditions, the same method of the analysis in this paper yields this relation should be changed to $\sigma n^{-1/2}\le Ch$. This suggests the importance of appropriate balance between the number of measurements and the finite element mesh sizes for solving PDEs with random observational data.

If the random variables $e_i, 1\le i\le n$, are also sub-Gaussian, we prove by resorting to the theory of empirical processes that
for any $z>0$,
\ben
\mathbb{P}\left(\|u-u_h\circ\Phi_h^{-1}\|_{L^2(\Om)}\ge \left[h^2|\ln h|\Delta(u,f,g_0)+|\ln h|(\sigma n^{-1/2})\right]z\right)\le 2e^{-Cz^2}.
\een
This implies that the probability of the random error $\|u-u_h\|_{L^2(\Om)}$ violating the error estimate in \eqref{p6} decays exponentially.

The layout of the paper is as follows. In section 2 we introduce our finite element formulation and derive an error estimate
based on the Babu\v{s}ka-Brezzi theory. In section 3 we study the random finite element error in terms of the expectation. In section 4 we show the stochastic convergence of our method when the random noise is sub-Gaussion. In section 5 we
report some numerical examples to confirm our theoretical analysis.

%%%%%%%%%%%%%%%%%%%%%%%%%%%%%%%%%%%%%%%%%%%%%%%%%%%%%%%%%%%%%%%
\section{The finite element method}
\setcounter{equation}{0}
%%%%%%%%%%%%%%%%%%%%%%%%%%%%%%%%%%%%%%%%%%%%%%%%%%%%%%%%%%%%%%%

We start by introducing the finite element meshes. Let $\cM_h$ be a mesh over $\Om$ consisting of curved triangles. We assume each element $K\in\cM_h$ has at most one curved edge and the edge of the element $K$ is curved only when its two vertices all lie on the boundary $\Ga$. For any $K\in\cM_h$, we denote $\tilde K$ the straight triangle which has the same vertices as $K$. We set $\Om_h=\cup_{K\in\cM_h}\tilde K$ and assume the mesh $\tilde\cM_h=\{\tilde K:K\in\cM_h\}$ over $\Om_h$ is shape regular and quasi-uniform: 
\be\label{a0}
h_{\tilde K}\le C\rho_{\tilde K},\ \ \forall K\in\cM_h,\ \ \ \ h_{\tilde K}\le Ch_{\tilde K'},\ \ \forall K,K'\in\cM_h,
\ee
where $h_{\tilde K}$ and $\rho_{\tilde K}$ are the diameter of $\tilde K$ and the diameter of the biggest circle inscribed in $\tilde K$. The finite element space for the field variable is then defined as
\ben
V_h=\{v_h\in C(\bar\Om_h): v_h|_{\tilde K}\in P_1(\tilde K),\forall \tilde K\in\tilde\cM_h\},
\een
where $P_1(\tilde K)$ is the set of the linear polynomials on $\tilde K$. As usual, we demote $h=\max_{\tilde K\in\tilde \cM_h}h_{\tilde K}$.

Let $\cE_h=\{K\cap \Ga: K\in\cM_h\}$ be the mesh of $\Ga$ which is induced from $\cM_h$. We assume that each element $E\in\cE_h$ is the image of the reference element $\hat E=[0,1]$ under a smooth mapping $F_E$. Since the boundary $\Ga$ is smooth, the argument in \cite[Theorem 4.3.3]{c78} implies that if the diameter of the element $h_E$ is sufficiently small,
\be
\|\hat DF_E\|_{L^\infty(\hat E)}\le Ch_E,\ \ \|D_TF_E^{-1}\|_{L^\infty(E)}\le Ch_E^{-1},\ \ \forall E\in\cE_h,\label{a1}
\ee
where $\hat D$ is the derivative in $\hat E$ and $D_T$ is the tangential derivative on $\Ga$. It is then obvious that there are constants $C_1,C_2$ independent of the mesh $\cM_h$ such that $C_1h\le h_E\le C_2h,\ \ \forall E\in\cE_h$. 
We use the following finite element space for the Lagrangian multiplier \cite{s95}:
\be\label{x1}
\qquad Q_h=\{\mu_h\in C(\Ga):\mu_h|_E=\hat\mu_h\circ F_E^{-1} \mbox{ for some }\hat\mu_h\in P_1(\hat E),\forall E\in\cE_h\},
\ee
where $P_1(\hat E)$ is the set of linear polynomials over $\hat E$.

We assume that the measurement locations $\mathbb{T}$ are uniformly distributed over $\Ga$ in the sense that \cite{u88} there exists a constant $B>0$ such that $\frac{s_{\max}}{s_{\min}} \leq B$, where
\ben
s_{\max}=\mathop {\rm sup}\limits_{x\in \Ga} \mathop {\rm inf}\limits_{1 \leq i \leq n} s(x,x_i) ,\ \ \ \
s_{\min}=\mathop {\rm inf}\limits_{1 \leq i \neq j \leq n} s(x_i,x_j),
\een
where $s(x,y)$ is the arc length between $x,y\in\Ga$. It is easy to see that there exist constants $B_1,B_2$ such that $B_1n^{-1}\le s_{\max}\le Bs_{\min}\le B_2n^{-1}$. 

We introduce the empirical inner product between the data and any function $v\in C(\Ga)$ as $\la g,v\ra_n=\sum^n_{i=1}\al_ig_iv(x_i)$. We also write $\la u,v\ra_n=\sum^n_{i=1}\al_iu(x_i)v(x_i)$ for any $u,v\in C(\Ga)$ and the empirical norm $\|u\|_n=(\sum_{i=1}^{n} \al_iu^2(x_i))^{1/2}$ for any $u\in C(\Ga)$. We remark that the empirical norm is in fact a semi-norm on $C(\Ga)$. The weights $\al_i$, $i=1,2\cdots,n$, are chosen such that $\la u,v\ra_n$ is a good quadrature formula for the inner product $\la u,v\ra$ that we describe now. 

Let $\mathbb{T}_E=\mathbb{T}\cap E$ be the measurement points in $E\in\cE_h$. Since the measurement 
locations are uniformly distributed, $n_E=\#\mathbb{T}_E\sim nh_E$. We further assume $t_{j,E}=
F_E^{-1}(x_j)$, $j=1,2\cdots,n_E$, are ordered as $0=t_{0,E}\le t_{1,E}<t_{2,E}<\cdots<t_{n_E,E}\le t_{n_E+1,E}=1$. 
We remark that the vertices of the element $E$ need not be at the measurement locations. Denote $\De t_{j,E}=t_{j,E}-t_{j-1,E}$, $j=1,2,\cdots,n_E+1$. We define the following quadrature formula 
\be
Q_{\mathbb{T}_E}(w)=\sum^{n_E}_{j=1}\omega_{j,E} w(t_{j,E}),\ \ \forall w\in C(\hat E),\label{b1}
\ee
where $\omega_{1,E}=\De t_{1,E}+\frac 12\De t_{2,E},\omega_{j,E}=\frac 12(\De t_{j,E}+\De t_{j+1,E}),j=2,\cdots,n_E-1,\omega_{n_E,E}=\frac 12\De t_{n_E,E}+\De t_{n_E+1,E}$.

\begin{lemma}\label{lem:2.1} There exists a constant $C$ independent of $\mathbb{T}_E$ such that
\ben
\left|\int_0^1w(t)dt-Q_{\mathbb{T}_E}(w)\right|&\le&C\int^1_0|w''(t)|dt+\frac 12\De t_{1,E}\int^{t_{1,E}}_{t_{0,E}}|w'(t)|dt\\
&+&\frac 12\De t_{n_E+1,E}\int^{t_{n_E+1,E}}_{t_{n_E,E}}|w'(t)|dt,\ \ \forall w\in W^{2,1}(\hat E).
\een
\end{lemma}

\begin{proof} We introduce the standard piecewise trapezoid quadrature rule 
\be\label{b5}
\tilde Q_{\mathbb{T}_E}(w)=\sum^{n_E+1}_{j=1}\De t_{j,E}\frac{w(t_{j-1,E})+w(t_{j,E})}2,
\ee
which is exact for linear functions. By the Bramble-Hilbert lemma we know that there exists a constant $C$ such that
\ben
\left|
\int^1_0 w(t)dt-\tilde Q_{\mathbb{T}_E}(w)\right|\le C\int_0^1|w''(t)|dt,\ \ \forall w\in W^{2,1}(\hat E).
\een
Now the lemma follows since
\be\label{b6}
& &Q_{\mathbb{T}_E}(w)-\tilde Q_{\mathbb{T}_E}(w)\\
&=&\frac 12\De t_{1,E}(w(t_{1,E})-w(t_{0,E}))+\frac 12\De t_{n_E+1,E}(w(t_{n_E,E})-w(t_{n_E+1,E})).\nn
\ee
This completes the proof.
\end{proof}

Now for any $v\in C(\Ga)$ we can define the following quadrature rule which defines the weights $\al_j$, $j=1,2,\cdots,n$, in the empirical inner product,
\be
\int_\Ga v ds&=&\sum_{E\in\cE_h}\int_0^1v(F_E(t))|F_E'(t)|dt\label{b2}\\
&\approx&\sum_{E\in\cE_h}\sum^{n_E}_{j=1}\omega_{j,E}|F_E'(t_{j,E})|v(x_j)\nn\\
&=&\sum^n_{j=1}\al_jv(x_j),\ \ \ \ \al_j=\sum_{E\in\cE_h,x_j\in \mathbb{T}_E}\omega_{j,E}|F_E'(t_{j,E})|.\nn
\ee
Since $\De t_{1,E}\le C\De t_{2,E}, \De t_{n_E+1,E}\le C\De t_{n_E,E}$, and $\De t_{j,E}/\De t_{k,E}\le C$ for any $j,k=2,\cdots,n_E$, because the points in $\mathbb{T}$ are uniformly distributed, we have $\omega_{j,E}\sim 1/n_E\sim 1/(nh_E)$. This implies by \eqref{a1} there exist constants $B_3,B_4$ such that
\be\label{b3}
B_3n^{-1}\le \al_j\le B_4 n^{-1},\ \ j=1,2,\cdots,n.
\ee

Let $y_j, j=1,2,\cdots,J$, be the nodes of the mesh $\cM_h$ on $\Ga$. For any $v_h\in V_h$, we define $\Pi_h v_h\in Q_h$ such that $\Pi_h v_h(y_j)=v_h(y_j), j=1,2,\cdots,J$. For any $E\in\cE_h$, let $\tilde E$ be the segment connecting
two vertices of $E$ and denote $F_{\tilde E}:\hat E\to\tilde E$ the affine mapping from the reference element $\hat E$ to $\tilde E$. Then $(\Pi_h v_h)(F_E(t))=v_h(F_{\tilde E}(t))$, $\forall t\in\hat E$.

Now we are in the position to define the finite element solution for the problem \eqref{p2}-\eqref{p3}. Given $f\in H^2(\Om)$
and the observation $g_i$ at $x_i$ of the boundary value $g_0(x_i)$, $i=1,2,\cdots n$, find $(u_h,\lam_h)\in V_h\times Q_h$ such that 
\be
(\na u_h,\na v_h)_{\Om_h}+\la\lam_h,\Pi_hv_h\ra_n&=&(I_hf,v_h)_{\Om_h},\ \ \forall v_h\in V_h,\label{a4}\\
\la \mu_h,\Pi_hu_h\ra_n&=&\la \mu_h,g\ra_n,\ \ \forall\mu_h\in Q_h,\label{a5}
\ee
where $(\cdot,\cdot)_{\Om_h}$ is the inner product of $L^2(\Om_h)$ and $I_h: C(\bar\Om)\to V_h$ is the standard Lagrange interpolation operator. The interpolation operator $I_h$ can be replaced by the Cl\'ement interplant \cite{c75} if the source $f$ is less regular. We remark that the computation in \eqref{a4}-\eqref{a5} does not involve any geometric representation of the boundary $\Ga$ due to the introduction of the quadrature.

Following \cite{p80, s95} we introduce the following mesh-dependent Sobolev norms
\ben
\|v\|_{1/2,h}^2=\sum_{E\in\cE_h}h_E^{-1}\|v\|_{L^2(E)}^2,\ \ \|v\|_{-1/2,h}^2=\sum_{E\in\cE_h}h_E\|v\|_{L^2(E)}^2,\ \ \forall v\in L^2(\Ga).
\een
We use the following norms for functions $v_h\in V_h$, $\mu_h\in Q_h$
\ben
\|v_h\|_{V_h}=\left(\|\na v_h\|_{L^2(\Om_h)}^2+\|\Pi_h v_h\|_{1/2,h}^2\right)^{1/2},\ \ \|\mu_h\|_{Q_h}=\|\mu_h\|_{-1/2,h}.
\een

We consider now the well-posedness of the discrete problem \eqref{a4}-\eqref{a5} in the framework of Babu\v{s}ka-Brezzi theory. We start from the following simple lemma.

\begin{lemma}\label{lem:2.2} There exists a constant $C$ such that
\ben
|\la 1,v\ra-\la 1,v\ra_n|\le C\sum_{E\in\cE_h}\int^1_0\left(h_E|\hat v''_E|+h_E^2|\hat v'_E|+h^3_E|\hat v_E|\right)dt,\ \ \forall v\in W^{2,1}(\Ga),
\een
where $\hat v_E(t)=v|_E(F_E(t))$ for any $E\in\cE_h$.
\end{lemma}

\begin{proof} We first note that since $\Ga$ is smooth, we have $|F_E''(t)|\le Ch_E^2, |F_E'''(t)|\le Ch_E^3$ for any $E\in\cE_h$.
Since
\ben
|\la 1,v\ra-\la 1,v\ra_n|\le \sum_{E\in\cE_h}\left|\int_E vds-Q_{\mathbb{T}_E}(\hat v_E(t)|F_E'(t)|)\right|,
\een
the lemma follows easily from Lemma \ref{lem:2.1} by taking $w=\hat v_E(t)|F_E'(t)|$ in each element $E\in\cE_h$. We omit the details.
\end{proof}

\begin{lemma}\label{lem:2.3}
Let $K_h=\{v_h\in V_h:\la \Pi_hv_h,\mu_h\ra_n=0,\ \ \forall\mu_h\in Q_h\}$. There exists a constant $\al>0$ independent of $h,n$
such that 
\ben
(\na v_h,\na v_h)_{\Om_h}\ge \al\|v_h\|_{V_h}^2,\ \ \forall v_h\in K_h.
\een
\end{lemma}

\begin{proof}  For simplicity we write $\tilde v_h=\Pi_hv_h\in Q_h$ for any $v_h\in V_h$. Then $\la \tilde v_h,\tilde v_h\ra_n=0$ for any $v_h\in K_h$. By Lemma \ref{lem:2.2} for $v=\tilde v_h^2$ we obtain after some simple computations
\ben
\qquad\ \  |\la \tilde v_h,\tilde v_h\ra-\la \tilde v_h,\tilde v_h\ra_n|\le Ch(\|\na v_h\|_{L^2(\Om_h)}^2+Ch^{1/2}\|\na v_h\|_{L^2(\Om_h)}\|\tilde v_h\|_{L^2(\Ga)}).
\een
Thus 
\ben
\|\tilde v_h\|_{1/2,h}^2\le Ch^{-1}\|\tilde v_h\|_{L^2(\Ga)}^2
&=&Ch^{-1}|\la \tilde v_h,\tilde v_h\ra-\la \tilde v_h,\tilde v_h\ra_n|\\
&\le&C\|\na v_h\|_{L^2(\Om_h)}^2+Ch\|\na v_h\|_{L^2(\Om_h)}\|\tilde v_h\|_{1/2,h}.
\een
This shows $\|\na v_h\|_{L^2(\Om_h)}^2\ge C\|\tilde v_h\|_{1/2,h}^2$ and completes the proof.
\end{proof}

\begin{lemma}\label{lem:2.4}
There exists constants $C_1,C_2>0,h_0>0$ independent of $h, n$ such that for $h\le h_0$,
\ben
C_1\|\mu_h\|_{L^2(\Ga)}\le\|\mu_h\|_n\le C_2\|\mu_h\|_{L^2(\Ga)},\ \ \forall\mu_h\in Q_h.
\een
\end{lemma}

\begin{proof} Since $\hat\mu_h(t)=\mu_h(F_E(t))$ is linear in $\hat E$ for any $E\in\cE_h$, we use Lemma \ref{lem:2.2} for $v=\mu_h^2$ to obtain
\ben
|\la\mu_h,\mu_h\ra-\la\mu_h,\mu_h\ra_n|\le C\sum_{E\in\cE_h}\int_{\hat E}h_E|\hat\mu_h|^2dt\le C\|\mu_h\|_{L^2(\Ga)}^2.
\een
This shows the right inequality. Next by definition we have
\be\label{b7}
\la\mu_h,\mu_h\ra_n=\sum_{E\in\cE_h}Q_{\mathbb{T}_E}(\hat\mu_h^2|F_E'|).
\ee
From \eqref{b5} and \eqref{a1} we know that for any $E\in\cE_h$,
\ben
\tilde Q_{\mathbb{T}_E}(\hat\mu_h^2|F_E'|)
&=&\frac 12\sum^{n_E}_{j=0}\int^{t_{j+1}}_{t_j}\left(\hat\mu_h(t_j)^2|F_E'(t_j)|+\hat\mu_h^2(t_{j+1})|F'_E(t_{j+1})|\right)dt\\
&\ge& Ch_E\sum^{n_E}_{j=0}\int^{t_j+1}_{t_j}\left(\hat\mu_h(t_j)^2+\hat\mu_h(t_{j+1})^2\right)dt\\
&\ge&Ch_E\sum^{n_E}_{j=0}\int^{t_j+1}_{t_j}|\hat\mu_h(t)|^2dt,
\een
where in the last inequality we have used the fact that $\hat\mu_h$ is linear in $\hat E$ and the Jensen inequality for convex functions. Thus $|\tilde Q_{\mathbb{T}_E}(\hat\mu_h^2|F_E'|)|\ge C\|\mu_h\|_{L^2(E)}^2,\forall E\in\cE_h$. On the other hand, by \eqref{b6} we have
\ben
|Q_{\mathbb{T}_E}(\hat\mu_h^2|F_E'|)-\tilde Q_{\mathbb{T}_E}(\hat\mu_h^2|F_E'|)|\le Ch_E\|\mu_h\|_{L^2(E)}^2.
\een
Therefore, by \eqref{b7}, $\|\mu_h\|_n\ge C\|\mu_h\|_{L^2(\Ga)}$ for sufficiently small $h$. This completes the proof.
\end{proof}

We have the following inf-sup condition for the empirical inner product.

\begin{lemma}\label{lem:2.5}
There exists a constant $h_0,\beta>0$ independent of $h, n$ such that for $h\le h_0$,
\ben
\sup_{v_h\in V_h\backslash\{0\}}\frac{\la \Pi_hv_h,\mu_h\ra_n}{\|v_h\|_{V_h}}\ge \beta\|\mu_h\|_{Q_h},\ \ \forall \mu_h\in Q_h.
\een
\end{lemma}

\begin{proof} The proof follows an idea in \cite{p79} where the inf-sup condition for the bilinear form $\la v_h,\mu_h\ra$ is proved.
Let $y_j, j=1,2,\cdots,J$, be the nodes of the mesh $\cM_h$ on $\Ga$ and denote $\psi_j, j=1,2,\cdots,J$, the corresponding nodal basis function of $V_h$.

For any $\mu_h\in Q_h$, we define $v_h(x)=\sum_{j=1}^J\mu_h(y_j)\psi_j(x)\in V_h$. It is easy to check that
\be\label{b8}
\|v_h\|_{V_h}^2\le C\sum_{j=1}^J|\mu_h(y_j)|^2\le Ch^{-1}\|\mu_h\|_{L^2(\Ga)}^2.
\ee
From the definition of $\Pi_h v_h\in Q_h$ we know that $\Pi_h v_h=\mu_h$ on $\Ga$. Thus by Lemma \ref{lem:2.4},
\ben
\la \Pi_h v_h,\mu_h\ra_n=\|\mu_h\|_n^2\ge C\|\mu_h\|_{L^2(\Ga)}^2.
\een
This completes the proof by using \eqref{b8}.
\end{proof}

By Lemma \ref{lem:2.4} we know that for any $v_h\in V_h$, $\mu_h\in Q_h$
\ben
|\la \Pi_hv_h,\mu_h\ra_n|\le C\|\Pi_hv_h\|_{L^2(\Ga)}\|\mu_h\|_{L^2(\Ga)}\le C\|v_h\|_{V_h}\|\mu_h\|_{Q_h}.
\een
Now by the standard Babu\v{s}ka-Brezzi theory (cf., e.g., \cite[Proposition 5.5.4]{bbf13}) we obtain the following theorem.

\begin{theorem}\label{thm:2.1} There exists a constant $h_0>0$ independent of $h,n$ such that for any $h\le h_0$, the discrete problem \eqref{a4}-\eqref{a5} has a unique solution $(u_h,\lam_h)\in V_h\times Q_h$. Moreover, for any $(u_I,\lam_I)\in V_h\times Q_h$, we have
\ben
\|u_h-u_I\|_{V_h}+\|\lam_h-\lam_I\|_{Q_h}\le C\sum_{i=1}^3M_{ih},
\een
where the errors $M_{1h},M_{2h}, M_{3h}$ are defined by
\ben
& &M_{1h}=\sup_{v_h\in V_h\backslash\{0\}}\frac{|(\na u_I,\na v_h)_{\Om_h}+\la\lam_I,\Pi_hv_h\ra_n-(I_hf,v_h)_{\Om_h}|}{\|v_h\|_{V_h}},\\
& &M_{2h}=\sup_{\mu_h\in Q_h\backslash\{0\}}\frac{|\la\mu_h,\Pi_hu_I-g_0\ra_n|}{\|\mu_h\|_{Q_h}},\ \ 
M_{3h}=\sup_{\mu_h\in Q_h\backslash\{0\}}\frac{|\la\mu_h,e\ra_n|}{\|\mu_h\|_{Q_h}}.
\een
\end{theorem}

%%%%%%%%%%%%%%%%%%%%%%%%%%%%%%%%%%%%%%%%%%%%%%%%%%%%%%%%%%%%%%%
\section{Convergence of the finite element method}
\setcounter{equation}{0}
%%%%%%%%%%%%%%%%%%%%%%%%%%%%%%%%%%%%%%%%%%%%%%%%%%%%%%%%%%%%%%%

We will use the Lenoir homeomorphism $\Phi_h:\Om_h\to\Om$ \cite{l86}. The mapping $\Phi_h$ is defined elementwise: for
any $\tilde K\in\tilde\cM_h$, $\Phi_h|_{\tilde K}=\Psi_K$ is a $C^2$-diffeomorphism from $\tilde K$ to $K$.  If no edge of $K$ belongs to $\pa\Om_h$, $\Psi_K=I$, the identity. If one edge $\tilde E$ of $\tilde K$ lies on $\pa\Om_h$ which corresponds to the 
curved edge $E$ of $K\in\cM_h$, $\Psi_K$ maps $\tilde E$ onto $E$ and $\Psi_K=I$, the identity, alongs the other two edges
of $\tilde K$. We need the following
properties of $\Psi_K$ from \cite{l86} in the following lemma.

\begin{lemma}\label{lem:3.1}
The following assertions are valid for any $\tilde K\in\tilde\cM_h$ and $K\in\cM_h$. \\
$1^\circ$ The mapping $\Psi_K:\tilde K\to K$ satisfies the following estimates
\ben
\|D^s(\Psi_K-I)\|_{L^\infty(\tilde K)}\le Ch^{2-s}, \ \ \forall s\le 2,\ \ \ \ \sup_{x\in\tilde K}|J(\Psi_K)(x)-1|\le Ch,
\een
where $J(\Psi_K)$ denotes the modulus of the Jacobi determinant of $\Psi_K$. \\
$2^\circ$ The mapping $\Psi_K^{-1}:K\to\tilde K$ satisfies
\ben
\|D^s(\Psi_K^{-1}-I)\|_{L^\infty(K)}\le Ch^{2-s}, \ \ \forall s\le 2,\ \ \ \ \sup_{x\in K}|J(\Psi_K^{-1})(x)-1|\le Ch.
\een
\end{lemma}

Let $r_h:L^1(\Om_h)\to V_h$ be the Cl\'ement interplant \cite{c78} which enjoys the following properties
\be
& &|v-r_h v|_{H^j(\tilde K)}\le Ch^{m-j}|v|_{H^m(\Delta_{\tilde K})},\ \ \forall \tilde K\in\tilde\cM_h,0\le j\le m, m=1,2,\label{c1}\\
& &|v-r_h v|_{H^j(e)}\le Ch^{m-j-1/2}|v|_{H^m(\Delta_e)},\ \ \forall e\in\tilde\cE_h,0\le j<m, m=1,2,\label{c2}
\ee
where $\tilde\cE_h$ is the set of all sides of the mesh $\tilde\cM_h$, and for any set $A\subset\Om_h$, $\Delta_{A}$ is the union of the elements surrounding $A$. We remark that \eqref{c1} is proved in \cite{c78} and \eqref{c2} is the consequence of
\eqref{c1} and the following scaled trace inequality
\ben
|v|_{L^2(e)}\le Ch^{-1/2}\|v\|_{L^2(\De_e)}+Ch^{1/2}\|\na v\|_{L^2(\De_e)},\ \ \forall v\in H^1(\Om_h).
\een
We will assume in this section the solution $u\in H^2(\Om)$ and thus $\lam\in H^{1/2}(\Ga)$. By the trace theorem, there exists
a function $\tilde\lam\in H^1(\Om)$ such that $\tilde\lam=\lam$ on $\Ga$ and $\|\tilde\lam\|_{H^1(\Om)}\le C\|\lam\|_{H^{1/2}(\Ga)}$. Now we define the following interpolation operator $R_h:L^2(\Om)\to L^2(\Om)$ 
\ben
R_hv=[r_h(v\circ\Phi_h)]\circ\Phi_h^{-1},\ \ \forall v\in L^2(\Om).
\een
We notice that similar interpolation functions are used in \cite{l86} where the Cl\'ement interpolation operator is replaced by the Lagrangian interpolation operator. The following theorem can be easily proved by using Lemma \ref{lem:3.1} and \eqref{c1}-\eqref{c2}.

\begin{lemma}\label{lem:3.2} For any $v\in H^2(\Om)$, we have $\|v-R_hv\|_{H^j(\Om)}\le Ch^{m-j}\|v\|_{H^m(\Om)}$, $
\|v-R_h v\|_{H^j(\Ga)}\le Ch^{m-j-1/2}\|v\|_{H^m(\Om)}$, $0\le j\le m-1, m=1,2$. 
\end{lemma}

For any $v_h\in V_h$, we denote $\check v_h=v_h\circ\Phi_h^{-1}$ which is a function defined in $\Om$. 
Let $\Om^*=\cup_{K\in\cM_h^*}K$, where $\cM_h^*$ is the set of all elements having one curved edge. Obviously, $|\Om^*|\le Ch$. By definition $\Phi_h=\Psi_K$ is identity for $K\in\cM_h\backslash \cM_h^*$, it is easy to check by using Lemma \ref{lem:3.1} that (cf. \cite[Lemma 8]{l86}) for any $v_h,w_h\in V_h$,
\be\label{c3}
\quad\ \  |(\na v_h,\na w_h)_{\Om_h}-(\na\check v_h,\na\check w_h)|\le Ch\|\check v_h\|_{H^1(\Om^*)}\|\check w_h\|_{H^1(\Om^*)}.
\ee
Now by the Poinc\'are inequality, it is easy to see that $\|v\|_{L^2(\Om)}\le C\|\na v\|_{L^2(\Om)}+C\|v\|_{1/2,h},\forall v\in H^1(\Om)$. Thus by \eqref{c3}
\ben
\|\check v_h\|_{H^1(\Om)}\le \|\na\check v_h\|_{L^2(\Om)}+C\|\check v_h\|_{1/2,h}\le C\|v_h\|_{V_h}+Ch^{1/2}\|\check v_h\|_{H^1(\Om)},
\een
which implies, for sufficiently small $h$, 
\be\label{c4}
\|\check v_h\|_{H^1(\Om)}\le C\|v_h\|_{V_h},\ \ \forall v_h\in V_h.
\ee

\begin{lemma}\label{lem:3.3} Let $(u,\lam)\in H^2(\Om)\times H^{1/2}(\Ga)$ be the solution of \eqref{p2}-\eqref{p3}. We have
\ben
\|u-u_h\circ\Phi_h^{-1}\|_{H^1(\Om)}+\|\lam-\lam_h\|_{-1/2,h}\le Ch\|u\|_{H^2(\Om)}+\sum_{i=1}^3M_{ih},
\een
where $M_{ih}$, $i=1,2,3$, are defined in Theorem \ref{thm:2.1} with $u_I=r_h(u\circ\Phi_h)\in V_h$ and $\lam_I=R_h\tilde\lam\in Q_h$.
\end{lemma}

\begin{proof} We first observe that by Lemma \ref{lem:3.2}
\ben
\qquad\|\lam-\lam_I\|_{L^2(\Ga)}\le Ch^{1/2}\|\tilde\lam\|_{H^1(\Om)}\le Ch^{1/2}\|\lam\|_{H^{1/2}(\Ga)}\le Ch^{1/2}\|u\|_{H^2(\Om)}.
\een
Notice that $\check u_h=R_h u$, we obtain by Lemma \ref{lem:3.2}, \eqref{c4}, and Theorem \ref{thm:2.1} that
\ben
& &\|u-u_h\circ\Phi_h^{-1}\|_{H^1(\Om)}+\|\lam-\lam_h\|_{-1/2,h}\\
&\le&\|u-R_hu\|_{H^1(\Om)}+\|\lam-R_h\tilde\lam\|_{-1/2,h}+C(\|u_h-u_I\|_{V_h}+\|\lam_h-\lam_I\|_{Q_h})\\
&\le&Ch\|u\|_{H^2(\Om)}+C\sum_{i=1}^3 M_{ih}.
\een
This completes the proof.
\end{proof}

\begin{lemma}\label{lem:3.4}
We have $M_{1h}\le Ch|\ln h|^{1/2}(\|u\|_{H^2(\Om)}+\|f\|_{H^2(\Om)})$.
\end{lemma}

\begin{proof} We first note that by \eqref{p2} we have
\ben
(\na u,\na\check v_h)+\la\lam,\check v_h\ra=(f,\check v_h),\ \ \forall v_h\in V_h.
\een
Now since $\Pi_h v_h=\check v_h$ on $\Ga$, for any $v_h\in V_h$, we have
\ben
& &|(\na u_I,\na v_h)_{\Om_h}+\la\lam_I,\Pi_hv_h\ra_n-(I_hf,v_h)_{\Om_h}|\\
&\le&|(f,\check v_h)-(I_hf,v_h)_{\Om_h}|+|(\na u_I,\na v_h)_{\Om_h}-(\na u,\na\check v_h)|+
|\la\lam,\check v_h\ra-\la\lam_I,\check v_h\ra_n|.
\een
Since $\Phi_h=\Psi_K$ is identity for $K\in\cM_h\backslash\cM_h^*$, we have
\ben
(f,\check v_h)-(I_hf,v_h)_{\Om_h}=\sum_{K\in\cM_h^*}\int_{\tilde K}((f\circ\Psi_K)v_hJ(\Psi_K)-I_h(f\circ\Psi_K) v_h)dx,
\een
which implies by using Lemma \ref{lem:3.1} that 
\ben
& &|(f,\check v_h)-(I_hf,v_h)_{\Om_h}|\\
&\le&Ch\|f\|_{L^2(\Om^*)}\|\check v_h\|_{L^2(\Om^*)}+Ch^2\|f\|_{H^2(\Om)}\|v_h\|_{L^2(\Om)}.
\een
Obviously, $\|f\|_{L^2(\Om^*)}\le Ch^{1/2}\|f\|_{L^\infty(\Om)}\le Ch^{1/2}\|f\|_{H^2(\Om)}$. Moreover, by the well-known embedding theorem \cite{r94} 
\ben
\|v\|_{L^p(\Om)}\le Cp^{1/2}\|v\|_{H^1(\Om)},\forall v\in H^1(\Om), \forall p>2,
\een 
we have
\ben
\|v\|_{L^2(\Om^*)}\le C|\Om^*|^{\frac 12-\frac 1p}p^{1/2}\|v\|_{H^1(\Om)}\le Ch^{\frac 12-\frac 1p}p^{1/2}\|v\|_{H^1(\Om)}.
\een
By taking $p=\ln (h^{-1})$ we obtain then 
\be\label{g2}
\|v\|_{L^2(\Om^*)}\le Ch^{1/2}|\ln h|^{1/2}\|v\|_{H^1(\Om)},\ \ \forall v\in H^1(\Om).
\ee
This implies 
\be\label{c5}
|(f,\check v_h)-(I_hf,v_h)_{\Om_h}|\le Ch^2|\ln h|^{1/2}\|f\|_{H^2(\Om)}\|v_h\|_{V_h}.
\ee
By Lemma \ref{lem:3.2}, \eqref{c3} and \eqref{c4} we have
\be\label{c6}
& &|(\na u_I,\na v_h)_{\Om_h}-(\na u,\na\check v_h)|\\
&\le&|(\na u_I,\na v_h)_{\Om_h}-
(\na\check u_I,\na\check v_h)|+|(\na (u-\check u_I),\na\check v_h)|\nn\\
&\le&Ch\|u\|_{H^2(\Om)}\|v_h\|_{V_h}.\nn
\ee
By using Lemma \ref{lem:2.2} one can prove
\be\label{g1}
|\la\check v_h,\check w_h\ra_n-\la\check v_h,\check w_h\ra|\le Ch\|v_h\|_{H^1(\Om_h)}\|w_h\|_{H^1(\Om_h)},\ \ \forall v_h,w_h\in V_h.
\ee
Thus
\ben
|\la\lam_I,\check v_h\ra_n-\la\lam_I,\check v_h\ra|&\le&Ch\|r_h(\tilde\lam\circ\Phi_h)\|_{H^1(\Om_h)}\|v_h\|_{H^1(\Om_h)}\\
&\le&Ch\|u\|_{H^2(\Om)}\|v_h\|_{V_h},
\een
which implies by using Lemma \ref{lem:3.2} that
\be\label{c7}
|\la\lam,\check v_h\ra-\la\lam_I,\check v_h\ra_n|\le Ch\|u\|_{H^2(\Om)}\|v_h\|_{V_h}.
\ee
The estimate for $M_{1h}$ now follows from \eqref{c5}, \eqref{c6} and \eqref{c7}. 
\end{proof}

\begin{lemma}\label{lem:3.5}
We have $M_{2h}\le Ch\|u\|_{H^2(\Om)}$.
\end{lemma}

\begin{proof} We first we observe that the argument in the proof of Lemma \ref{lem:2.1} implies that
\ben
\left|\int^1_0w(t)dt-Q_{\mathbb{T}_E}(w)\right|\le C\|w'\|_{L^2(\hat E)},\ \ \forall w\in H^1(\hat E).
\een
For any $v\in H^1(\Ga)$, by taking $w(t)=\hat v_E(t)|F_E'(t)|$ in each element $E\in\cE_h$, where $\hat v_E(t)=v|_E(F_E(t))$,
we know that 
\ben
|\la 1,v\ra-\la 1,v\ra_n|\le C\sum_{E\in\cE_h}(h_E\|\hat v_E'\|_{L^2(\hat E)}+h_E^2\|\hat v_E\|_{L^2(\hat E)}).
\een
We use the above inequality for $v=\mu_h\vp$, where $\vp=u-\check u_I$ in $\Ga$, to obtain
\ben
|\la\mu_h,\vp\ra-\la\mu_h,\vp\ra_n|\le C\sum_{E\in\cE_h}\|\hat\mu_h\|_{L^2(\hat E)}(h_E\|\hat\vp_E\|_{L^2(\hat E)}+h^2_E\|\hat\vp_E'\|_{L^2(\hat E)}),
\een
where we have used the fact $\|\hat\mu_h\|_{W^{1,\infty}(\hat E)}\le C\|\hat\mu_h\|_{L^2(\hat E)}$ since $\hat\mu_h\in P_1(\hat E)$. This implies by using Lemma \ref{lem:3.2} again
\ben
& &|\la\mu_h,u-\check u_I\ra-\la\mu_h,u-\check u_I\ra_n|\\
&\le&C\|\mu_h\|_{L^2(\Ga)}(\|u-R_hu\|_{L^2(\Ga)}+h|u-R_h u|_{H^1(\Ga)})\\
&\le&Ch^{3/2}\|u\|_{H^2(\Om)}\|\mu_h\|_{L^2(\Ga)}.
\een
This completes the proof.
\end{proof}

The following theorem shows the convergence of the finite element solution in the sense of expectation.

\begin{theorem}\label{thm:3.1} We have
\ben
& &\mathbb{E}\left[\|u-u_h\circ\Phi_h^{-1}\|_{H^1(\Om)}+h^{1/2}\|\lam-\lam_h\|_{L^2(\Ga)}\right]\\
&\le&Ch|\ln h|^{1/2}(\|u\|_{H^2(\Om)}+\|f\|_{H^2(\Om)})+Ch^{-1}(\sigma n^{-1/2}).
\een
\end{theorem}

\begin{proof} By Lemmas \ref{lem:3.3}-\ref{lem:3.5} we are left to estimate $\mathbb{E}[M_{3h}]$. We first observe that
\be\label{g3}
\mathbb{E}\left[\sup_{\mu_h\in Q_h\backslash\{0\}}\frac{|\la\mu_h,e\ra_n|^2}{\|\mu_h\|_{Q_h}^2}\right]\le Ch^{-1}\mathbb{E}\left[\sup_{\mu_h\in Q_h\backslash\{0\}}\frac{|\la\mu_h,e\ra_n|^2}{\|\mu_h\|_{L^2(\Ga)}^2}\right].
\ee
Let $N_h$ be the dimension of $Q_h$ and $\{\psi_j\}_{j=1}^{N_h}$ be the orthonormal basis of $Q_h$ in the $L^2(\Ga)$ inner product. Then for any $\mu_h=\sum^{N_h}_{j=1}(\mu_h,\psi_j)\psi_j$, by Cauchy-Schwarz inequality and \eqref{b3}
\ben
|\la\mu_h,e\ra_n|^2\le\frac C{n^2}\|\mu_h\|_{L^2(\Ga)}^2\sum^{N_h}_{j=1}\left(\sum^n_{i=1}e_i\psi_j(x_i)\right)^2.
\een
Since $e_i,i=1,2,\cdots,n$, are independent and identically random variables, we have 
\ben
\mathbb{E}\left[\sup_{\mu_h\in Q_h\backslash\{0\}}\frac{|\la\mu_h,e\ra_n|^2}{\|\mu_h\|_{L^2(\Ga)}^2}\right]\le C\frac{\sigma^2}{n^2}
\sum^{N_h}_{j=1}\sum^n_{i=1}\psi_j(x_i)^2.
\een
Since the number of measurement points in $E$, $
\#\mathbb{T}_E\le Cnh_E$ and $N_h\le Ch^{-1}$, we obtain by using the inverse estimate that
\ben
\sum^{N_h}_{j=1}\sum^n_{i=1}\psi_j(x_i)^2\le Cnh\sum^{N_h}_{j=1}\sum_{E\in\mathcal{E}_h}\|\psi_j\|_{L^\infty(E)}^2\le CN_h n\le Cnh^{-1}.
\een
Therefore
\be\label{x3}
\mathbb{E}\left[\sup_{\mu_h\in Q_h\backslash\{0\}}\frac{|\la\mu_h,e\ra_n|^2}{\|\mu_h\|_{L^2(\Ga)}^2}\right]\le Ch^{-1}(\sigma^2n^{-1}).
\ee
This, together with \eqref{g3}, yields
\ben
\mathbb{E}\left[\sup_{\mu_h\in Q_h\backslash\{0\}}\frac{|\la\mu_h,e\ra_n|^2}{\|\mu_h\|_{Q_h}^2}\right]\le Ch^{-2}(\sigma^2n^{-1}).
\een
This completes the proof.
\end{proof}

The following lemma will be useful in deriving the improved estimate for $\|u-u_h\circ\Phi_h^{-1}\|_{L^2(\Om)}$.

\begin{lemma}\label{lem:3.6}
We have
\ben
\mathbb{E}\left[\sup_{\mu_h\in Q_h\backslash\{0\}}\frac{|\la e,\mu_h\ra_n|}{\|\mu_h\|_{H^{1/2}(\Ga)}}\right]\le C|\ln h|(\sigma n^{-1/2}).
\een
\end{lemma}

\begin{proof}
Let $h_0=h\le 1$ and $h_i=h^{(p+1-i)/(p+1)}$ for $1\le i\le p$, where $p\ge 1$ is an integer to be determined later.
Obviously $h_i\le h_{i+1}$, $0\le i\le p$. Let $\mathcal{E}_{h_i}$ be a uniform mesh over the boundary $\Ga$ and $Q_{h_i}$
the finite element space defined in \eqref{x1} over the mesh $Q_{h_i}$. Let $\{y_{h_i}^k\}^{N_{h_i}}_{k=1}$ be the nodes of 
the mesh $\mathcal{E}_{h_i}, i=0,\cdots,p$. We introduce the following Cl\'ement-type interpolation operator $\pi_{h_i}:L^1(\Ga)
\to Q_{h_i}$ such that for any $v\in L^1(\Ga)$,
\ben
(\pi_{h_i}v)(y_{h_i}^k)=\frac 1{|S(y_{h_i}^k)|}\int_{S(y_{h_i}^k)}v(x) ds(x),\ \ 1\le k\le N_{h_i},
\een
where $S(y_{h_i}^k)$ is the union of the two elements sharing the common node $y_{h_i}^k$. It is easy to show by scaling argument that
\ben
\|v-\pi_{h_i} v\|_{L^2(\Ga)}\le h_i^m\|v\|_{H^m(\Ga)},\ \ \forall v\in H^1(\Ga), m=0,1.
\een
Thus by the theory of real interpolation of Sobolev spaces, e.g., \cite[Proposition 12.1.5]{bs94},
\be\label{y1}
\|v-\pi_{h_i} v\|_{L^2(\Ga)}\le Ch_i^{1/2}\|v\|_{H^{1/2}(\Ga)},\ \ \forall v\in H^{1/2}(\Ga).
\ee
Now we introduce the telescope sum
\be\label{y2}
\mu_h=\sum^{p-1}_{i=0}(\mu_{h_i}-\mu_{h_{i+1}})+\mu_{h_p},\ \ \forall\mu_h\in Q_h=Q_{h_0},
\ee
where $\mu_{h_i}=\pi_{h_i}\mu_h\in Q_{h_i}$, $0\le i\le p+1$. By \eqref{y1}
\be\label{z0}
\|\mu_{h_i}-\mu_{h_{i+1}}\|_{L^2(\Ga)}\le Ch_{i+1}^{1/2}\|\mu_h\|_{H^{1/2}(\Ga)}.
\ee
Then the same argument in proving \eqref{x3} implies
\ben
& &\mathbb{E}\left[\sup_{\mu_h\in Q_h\backslash\{0\}}\frac{|\la e,\mu_{h_i}-\mu_{h_{i+1}}\ra_n|}{\|\mu_h\|_{H^{1/2}(\Ga)}}\right]\le Ch_{i+1}^{1/2}h_i^{-1/2}(\sigma n^{-1/2}),\\
& &\mathbb{E}\left[\sup_{\mu_h\in Q_h\backslash\{0\}}\frac{|\la e,\mu_{h_p}\ra_n|}{\|\mu_h\|_{H^{1/2}(\Ga)}}\right]
\le Ch_p^{-1/2}(\sigma n^{-1/2}).
\een
By \eqref{y2} we then obtain
\ben
\mathbb{E}\left[\sup_{\mu_h\in Q_h\backslash\{0\}}\frac{|\la e,\mu_h\ra_n|}{\|\mu_h\|_{H^{1/2}(\Ga)}}\right]
\le C(p+1)h^{-\frac 1{2(p+1)}}(\sigma n^{-1/2}).
\een
This completes the proof by taking the integer $p$ such that $p< |\ln h|\le p+1$.
\end{proof}

\begin{theorem}\label{thm:3.2} We have
\ben
& &\mathbb{E}\left[\|u-u_h\circ\Phi_h^{-1}\|_{L^2(\Om)}\right]\\
&\le&Ch^2|\ln h|(\|u\|_{H^2(\Om)}+\|f\|_{H^2(\Om)}+\|g_0\|_{H^2(\Ga)})+C|\ln h|(\sigma n^{-1/2}).
\een
\end{theorem}

\begin{proof} Let $(w,p)\in H^1(\Om)\times H^{-1/2}(\Ga)$ be the solution of the following problem
\be
(\na w,\na v)+\la p,v\ra&=&(u-\check u_h,v),\ \ \forall v\in H^1(\Om),\label{h1}\\
\la\mu,w\ra&=&0,\ \ \forall\mu\in H^{-1/2}(\Ga).\label{h2}
\ee
By the regularity theory of elliptic equations, $(w,p)\in H^2(\Om)\times H^1(\Om)$ and satisfies
\be\label{h3}
\|w\|_{H^2(\Om)}+\|p\|_{H^1(\Om)}\le C\|u-\check u_h\|_{L^2(\Om)}.
\ee
Let $w_I=I_h(w\circ\Phi_h)\in V_h$ be the Lagrange interpolation of $w\in H^2(\Om)$ and $p_I=r_h(p\circ\Phi_h)\in V_h$ be the Cl\'ement interpolation of $p\in H^1(\Om)$. By \eqref{h2} we know that $w=0$ on $\Ga$ and consequently, $w_I=0$ on $\Ga_h$, $\check w_I=w_I\circ\Phi_h^{-1}=0$ on $\Ga$. Now by using \eqref{p2}-\eqref{p3}, \eqref{a4}-\eqref{a5} we obtain
\be\label{h4}
& &\|u-\check u_h\|_{L^2(\Om)}^2\\
&=&(\na(w-\check w_I),\na (u-\check u_h))+\la p-\check p_I,u-\check u_h\ra\nn\\
&+&[(f,\check w_I)-(I_hf,\check w_I)_{\Om_h}]+[(\na w_I,\na u_h)_{\Om_h}-(\na\check w_I,\na\check u_h)]\nn\\
&+&[\la\check p_I,u-\check u_h\ra-\la\check p_I,u-\check u_h\ra_n]-\la \check p_I,e\ra_n\nn\\
&:=&{\rm I}+\cdots+{\rm VI}.\nn
\ee
By Lemma \ref{lem:3.1} and \eqref{h3} we have 
\be\label{h5}
\qquad|{\rm I}|+|{\rm II}|\le Ch\|u-\check u_h\|_{H^1(\Om)}\|u-\check u_h\|_{L^2(\Om)}.
\ee
By \eqref{c5} and \eqref{h3}
\be\label{h6}
|{\rm III}|&\le&Ch^2|\ln h|^{1/2}\|f\|_{H^2(\Om)}\|w_I\|_{V_h}\\
&\le&Ch^2|\ln h|^{1/2}\|f\|_{H^2(\Om)}\|u-\check u_h\|_{L^2(\Om)}.\nn
\ee
Since $\Phi_h|_K=I$ for $K\in\cM_h\backslash\cM_h^*$, by \eqref{c3}, Lemma \ref{lem:3.1} and \eqref{h3} we have
\ben
|{\rm IV}|&\le&Ch\|\check w_I\|_{H^1(\Om^*)}\|\check u_h\|_{H^1(\Om^*)}
\een
Now by using \eqref{g2}, Lemma \ref{lem:3.1}, and \eqref{h3}, we have
\ben
& &\|\check u_h\|_{H^1(\Om^*)}\le\|u-\check u_h\|_{H^1(\Om)}+Ch^{1/2}|\ln h|^{1/2}\|u\|_{H^2(\Om)},\\
& &\|\check w_I\|_{H^1(\Om^*)}\le Ch\|u-\check u_h\|_{L^2(\Om)}+Ch^{1/2}|\ln h|^{1/2}\|u-\check u_h\|_{L^2(\Om)}.
\een
This implies
\be\label{h7}
|{\rm IV}|\le C\left[h\|u-\check u_h\|_{H^1(\Om)}+h^2|\ln h|\|u\|_{H^2(\Om)}\right]\|u-\check u\|_{L^2(\Om)}.
\ee
To estimate the term ${\rm V}$ we first use the triangle inequality
\ben
|{\rm V}|\le|\la\check p_I,u-\check u_I\ra-\la\check p_I,u-\check u_I\ra_n|+|\la\check p_I,u_I-\check u_h\ra-\la\check p_I,\check u_I-\check u_h\ra_n|
\een
By using Lemma \ref{lem:2.2} for $v=\check p_I(u-\check u_I)$ one obtains easily
\ben
|\la\check p_I,u-\check u_I\ra-\la\check p_I,u-\check u_I\ra_n|&\le&Ch^2\|g_0\|_{H^2(\Ga)}\|\check p_I\|_{L^2(\Ga)}\\
&\le&Ch^2\|g_0\|_{H^2(\Ga)}\|u-\check u_h\|_{L^2(\Om)},
\een
where we have used the estimate $\|\check p_I\|_{L^2(\Ga)}\le C\|p\|_{H^1(\Om)}\le C\|u-\check u_h\|_{L^2(\Om)}$.
By \eqref{g1} and \eqref{h3} we have  
\ben
|\la\check p_I,u_I-\check u_h\ra-\la\check p_I,\check u_I-\check u_h\ra_n|
&\le&Ch\|p_I\|_{H^1(\Om_h)}\|u_I-u_h\|_{H^1(\Om_h)}\\
&\le&Ch\|u_I-u_h\|_{H^1(\Om_h)}\|u-\check u_h\|_{L^2(\Om)}.
\een
Thus
\be\label{h8}
\qquad|{\rm V}|\le Ch^2(h^{-1}\|u-\check u_h\|_{H^1(\Om)}+\|u\|_{H^2(\Om)}+\|g_0\|_{H^2(\Ga)})\|u-\check u_h\|_{L^2(\Om)}.
\ee
By inserting \eqref{h5}-\eqref{h8} into \eqref{h4} we obtain finally
\be\label{h9}
\|u-\check u_h\|_{L^2(\Om)}&\le&Ch^2|\ln h|\left[\|u\|_{H^2(\Om)}+\|f\|_{H^2(\Om)}+\|g_0\|_{H^2(\Ga)}\right]\\
&+&Ch\|u-\check u_h\|_{H^1(\Om)}+\sup_{\mu_h\in Q_h\backslash\{0\}}\frac{|\la e,\mu_h\ra_n|}{\|\mu_h\|_{H^{1/2}(\Ga)}}.\nn
\ee
The lemma now follows from Theorem \ref{thm:3.1} and Lemma \ref{lem:3.6}.
\end{proof}

%%%%%%%%%%%%%%%%%%%%%%%%%%%%%%%%%%%%%%%%%%%%%%%%%%%%%%%%%%%%%%%
\section{Sub-Gaussian random errors}
\setcounter{equation}{0}
%%%%%%%%%%%%%%%%%%%%%%%%%%%%%%%%%%%%%%%%%%%%%%%%%%%%%%%%%%%%%%%

In this section, we will study the convergence of our finite element method when the random errors added to the boundary data are sub-Gaussian. We will use the theory of empirical processes \cite{u88, vw96}. 

\begin{definition}\label{def:4.1}
A random variable $X$ is called sub-Gaussian with parameter $\sigma$ if
\ben
\mathbb{E}[e^{\lambda (X-\mathbb{E}[X])}] \leq e^{\sigma^2\lambda^2/2}, ~~\forall~~ \lambda \in \R.
\een
\end{definition}
 
The following definition on the Orilicz $\psi_2$-norm will be used in our analysis. 
\begin{definition}\label{def:4.2}
Let $\psi_2=e^{x^2}-1$ and $X$ be a random variable. The $\psi_2$ norm of $X$ is defined as
\ben
\|X\|_{\psi_2} = \inf\left\{C>0: ~ \mathbb{E}\left[\psi_2\left(\frac{|X|}{C}\right)\right] \leq 1\right\}.
\een
\end{definition}

It is known that \cite[Lemma 2.2.1]{vw96} if $\|X\|_{\psi_2} \leq K$, then
\be\label{d1}
\mathbb{P}(|X|> z) \leq 2 \exp\left(-\frac{z^2}{K^2}\right),\ \forall \ z>0.
\ee
Inversely, if
\be\label{d2}
\mathbb{P}(|X|> z) \leq C \exp\left(-\frac{z^2}{K^2}\right),\ \forall \ z>0,
\ee
then $\|X\|_{\psi_2} \leq \sqrt {1+C} K$.

\begin{definition}\label{def:4.3}
Let $(T,d)$ be a semi-metric space, a stochastic process $\{X_t:~ t\in T\}$ is called a sub-Gaussian process respect to the semi-metric $d$, if
\ben
\mathbb{P}(|X_t-X_s|>z) \leq 2 \exp\left(-\frac 12\frac{z^2}{d^2(t,s)}\right),~\forall~s,t\in T, \ z>0.
\een
\end{definition}

For a semi-metric space $(T,d)$, the covering number $N(\vep,T,d)$ is the minimum number of $\vep$-balls that covers $T$. A set is called $\vep$-separated if the distance of any two points in the set is strictly greater than $\vep$. The packing number $D(\vep,T,d)$ is the maximum number of $\vep$-separated points in $T$. It is easy to check that \cite[P.98]{vw96}
\be\label{f1}
N(\vep,T,d)\le D(\vep,T,d)\le N(\frac\vep 2,T,d).
\ee
The following maximum inequality can be found in \cite[Section 2.2.1]{vw96}.

\begin{lemma}\label{lem:4.1}
If $\{X_t: t \in T\}$ is a separable sub-Gaussian process respect to the semi-metric $d$, then
\ben
\|\sup_{s,t\in T} |X_t-X_s|\|_{\psi_2} \leq K\int_0^{\mbox{\scriptsize\rm diam}\,T} \sqrt{D(\vep,T,d)} \,d\vep.
\een
Here $K>0$ is some constant.
\end{lemma}

The following lemma provides the estimate of the covering number for finite dimensional subsets \cite[Corollary 2.6]{g00}.
\begin{lemma}\label{lem:4.2}
Let $G$ be a finite dimensional subspace of $L^2(D)$ of dimension $N>0$ and $G_R=\{f\in G: \|f\|_{L^2(D)}\le R\}$. Then
\ben
N(\vep,G_R,\|\cdot\|_{L^2(D)})\le (1+4R/\vep)^N,\ \ \forall \vep>0.
\een
\end{lemma}

\begin{theorem}\label{thm:4.1}We have
\ben
& &\|\,\|u-u_h\circ\Phi_h^{-1}\|_{H^1(\Om)}\,\|_{\psi_2}+h^{1/2}\|\,\|\lam-\lam_h\|_{L^2(\Ga)}\,\|_{\psi_2}\\
&\le&Ch|\ln h|^{1/2}(\|u\|_{H^2(\Om)}+\|f\|_{H^2(\Om)})+Ch^{-1}(\sigma h^{-1/2}).
\een
\end{theorem}

\begin{proof} By Lemmas \ref{lem:3.3}-\ref{lem:3.5} we are left to estimate $\|M_{3h}\|_{\psi_2}$. 
Let $F_h=\{\mu_h\in Q_h:\mbox{ }\|\mu_h\|_{L^2(\Ga)}\leq 1\}$, then
\be\label{f2}
\|M_{3h}\|_{\psi_2}\leq h^{-1/2}\|\sup_{\mu_h\in  F_h}\la |\mu_h,e\ra_n| \|_{\psi_2}.
\ee
For any $\mu_h\in F_h$, denote by $E_n(\mu_h)=\la\mu_h,e\ra_n$. Then $E_n(\mu_h)-E_n(\mu_h')=\sum^n_{i=1}c_ie_i$, where $c_i=\alpha_i(\mu_h-\mu'_h)(x_i)$, $i=1,2,\cdots,n$. For any $\lambda >0 $, since $\al_i\le B_4n^{-1}$ by \eqref{b3},
\ben
\mathbb{E}\left[e^{\lambda\sum_{i=1}^n c_i e_i}\right]&\leq& e^{\frac{1}{2}\lambda^2 \sigma^2 \sum_{i=1}^n c_i^2} 
\le e^{\frac{1}{2}B_4\lambda^2 \sigma^2 n^{-1}\|\mu_h-\mu_h'\|^2_n} = e^{\frac{1}{2}\sigma_1^2\lambda^2},
\een
where $\sigma_1=B_4\sigma n^{-1/2}\|\mu_h-\mu_h'\|_n$. Thus $E_n(\mu_h)-E_n(\mu_h')$ is a sub-Gaussian process with the parameter $\sigma_1$. This implies by \eqref{d1} that 
\ben
\mathbb{P}(|E_n(\mu_h-\mu'_h)|>z) \leq 2e^{-z^2/2\sigma_1^2}, \ \ \forall z>0.
\een
Thus $E_n(\mu_h)$ is a sub-Gaussian random process with respect to the semi-distance $d(\mu_h,\mu_h')=\|\mu_h-\mu_h'\|_n^*$, where $\|\mu_h\|^*_n = B_4\sigma n^{-1/2}\|\mu_h\|_n$.

By Lemma \ref{lem:2.4} we know that the diameter of $F_h$ in terms of the semi-distance $d$ is bounded by $2C_2B_4(\sigma n^{-1})$. By maximal inequality in Lemma \ref{lem:4.1} and \eqref{f1} we have
\ben
\|\sup_{\mu_h\in  F_h}|\la\mu_h,e\ra_n|\|_{\psi_2} & \leq & K\int_0^{2C_2B_4\sigma n^{-1/2}}\sqrt{\log N(\frac{\varepsilon}{2},F_h,\|\cdot\|_n^*)}\,d\vep \\
& = & K\int_0^{2C_2B_4\sigma n^{-1/2}}\sqrt{\log N(\frac{\varepsilon}{2B_4\sigma n^{-1/2}},F_h,\|\cdot\|_n)}\,d\varepsilon.
\een
By Lemma \ref{lem:2.4} and Lemma \ref{lem:4.2} we know that for any $\de>0$,
\ben
N(\de,F_h,\|\cdot\|_n)\le N(C_1^{-1}\de,F_h,\|\cdot\|_{L^2(\Ga)})\le (1+4C_1/\de)^{N_h},
\een
where $N_h$ is the dimension of $Q_h$ which is bounded by $Ch^{-1}$. Therefore,
\be\label{z1}
\|\sup_{\mu_h\in  F_h}|\la\mu_h,e\ra_n|\|_{\psi_2}
&\le&Ch^{-1/2}\int_0^{2C_2B_4\sigma n^{-1/2}}\sqrt{\log\left(1+\frac{C\sigma n^{-1/2}}{\vep}\right)}\,d\varepsilon\nn\\
&\le&Ch^{-1/2}(\sigma n^{-1/2}).
\ee
This shows $\|M_{3h}\|_{\psi_2}\le Ch^{-1}(\sigma n^{-1/2})$ by \eqref{f2}.
\end{proof}

By \eqref{d2}, Theorem \ref{thm:4.1} implies that the probability of the $H^1$-finite element error violatingf the convergence
order $O(h|\ln h|^{1/2}(\|u\|_{H^2(\Om)}+\|f\|_{H^2(\Om)})+h^{-1}(\sigma n^{-1/2}))$ decays exponentially.

\begin{theorem}\label{thm:4.2} We have
\ben
& &\|\,\|u-u_h\circ\Phi_h^{-1}\|_{L^2(\Om)}\,\|_{\psi_2}\\
&\le&Ch^2|\ln h|(\|u\|_{H^2(\Om)}+\|f\|_{H^2(\Om)}+\|g_0\|_{H^2(\Ga)})+C|\ln h|(\sigma n^{-1/2}).
\een
\end{theorem}

\begin{proof} Let $G_h=\{\mu_h\in Q_h:\|\mu_h\|_{H^{1/2}(\Ga)}\le 1\}$. By \eqref{h9} we are left to show
\be\label{z2}
\|\sup_{\mu_h\in G_h}|\la \mu_h,e\ra_n|\|_{\psi_2}\le C|\ln h|(\sigma n^{-1/2}).
\ee
We again use the telescope sum in \eqref{y2} and obtain
\be\label{z3}
& &\|\sup_{\mu_h\in G_h}|\la \mu_h,e\ra_n|\|_{\psi_2}\\
&\le&\sum^{p-1}_{i=0}\|\sup_{\mu_h\in G_h}|\la \mu_{h_i}-\mu_{h_{i+1}},e\ra_n|\|_{\psi_2}+\|\sup_{\mu_h\in G_h}|\la \mu_{h_p},e\ra_n|\|_{\psi_2}.\nn
\ee 
By the same argument in proving \eqref{z1} and using \eqref{z0} we have
\ben
& &\|\sup_{\mu_h\in G_h}|\la \mu_{h_i}-\mu_{h_{i+1}},e\ra_n|\|_{\psi_2}\le Ch_{i+1}^{1/2}(h_i^{-1/2}+h_{i+1}^{-1/2})(\sigma n^{-1/2}),\\
& &\|\sup_{\mu_h\in G_{h_p}}|\la \mu_h,e\ra_n|\|_{\psi_2}\le Ch_p^{-1/2}(\sigma n^{-1/2}).
\een 
Inserting the estimates to \eqref{z3} shows \eqref{z2} by taking $p$ such that $|\ln h|< p\le |\ln h|+1$.
\end{proof}

By \eqref{d2}, Theorem \ref{thm:4.2} implies that the probability of the $L^2$-finite element error violating the convergence
order $O(h^2|\ln h|(\|u\|_{H^2(\Om)}+\|f\|_{H^2(\Om)}+\|g_0\|_{H^2(\Ga)})+|\ln h|(\sigma n^{-1/2}))$ decays exponentially.

%%%%%%%%%%%%%%%%%%%%%%%%%%%%%%%%%%%%%%%%%%%%%%%%%%%%%%%%%%%%%%%
\section{Numerical examples}
\setcounter{equation}{0}
%%%%%%%%%%%%%%%%%%%%%%%%%%%%%%%%%%%%%%%%%%%%%%%%%%%%%%%%%%%%%%%

In this section, we show several numerical experiments to verify the theoretical analysis in this paper. 
The analyses in section 3 and section 4 suggest that the optimal convergence rate can be achieved by taking $n=O(h^{-4})$. For the examples below, we take the exact solution $u_0=\sin(5x+1)\sin(5y+1)$.

\begin{example}\label{numerical.1}
We take $\Omega=(0,1)\times(0,1)$. We construct the finite element mesh by first dividing the domain into $h^{-1}\times h^{-1}$ uniform rectangles and then connecting the lower left and upper right
angle. We set $\{x_i\}_{i=1}^n$ being uniformly distributed on $\Ga$, and $e_i,~ i=1,2,\cdots, n$, being independent normal random variables with variance $\sigma=2$.  We take different $n=h^{-i}$, $i=1,2,3,4$. Figure \ref{example.1} shows
the convergence rate of the error in the $H^1$ and $L^2$ norm for each choice of $n$. Table \ref{tab.1} show the convergence rate $\alpha$ in the $H^1$ norm and the convergence rate $\beta$ in the $L^2$ norm.

We observe the numerical results confirm our theoretical analysis. The optimal convergence rate is achieved when choosing $n=h^{-4}$ while the other choices do not achieve optimal convergence. For example, when $n=h^{-2}$, the $L^2$ error is approximately $O(h^1)$ and no convergence for the $H^1$ error.
\end{example}

\begin{figure}
\subfigure[$H^1$ convergence]{\includegraphics[width=6cm]{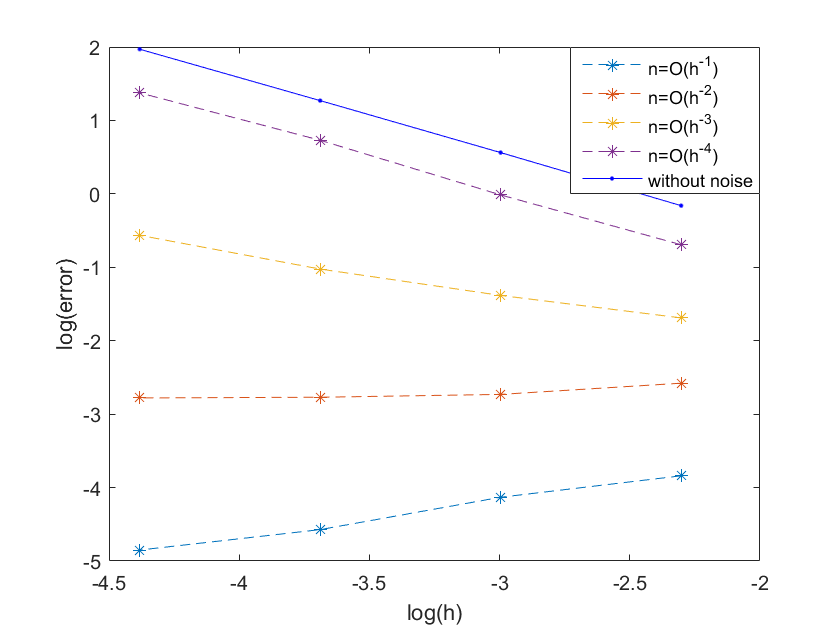}}
\subfigure[$L^2$ convergence]{\includegraphics[width=6cm]{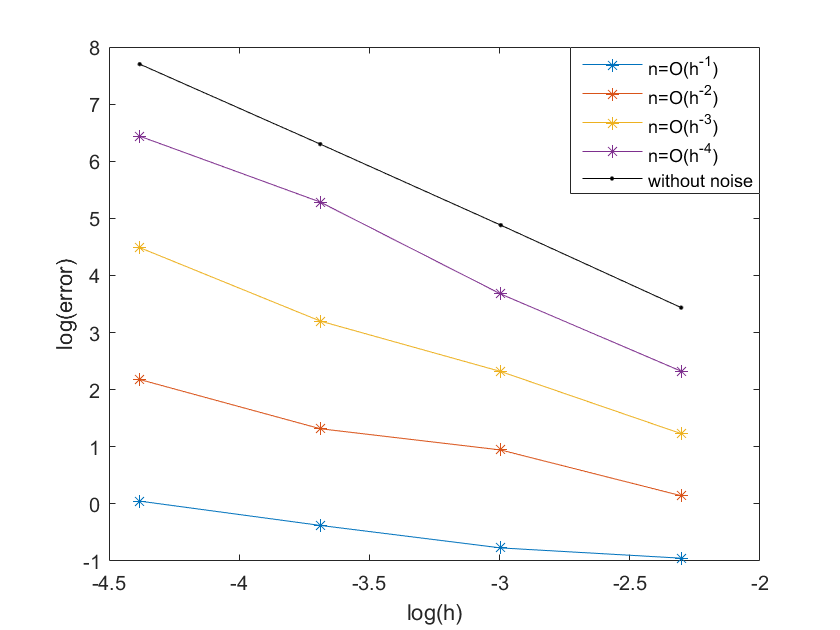}}
\caption{The log-log plot of the convergence rate on the unit square.}
\label{example.1}
\end{figure}

\begin{center}
\begin{table}
\begin{tabular}{ | c | c | c | c | c | c |}
\hline
$n$ & $h$ & $H^1$ error & $\alpha$ & $L^2$ error & $\beta$ \\ \hline \hline
\multirow{2}{*}
{$n=h^{-1}$} 
 & 0.1000 & 8.8686   &             & 0.3978 &\\ \cline{2-6}
 % & 0.0500 & 10.2980 & 0.2156 & 0.2621 & -0.6016\\ \cline{2-6}
 % & 0.0250 & 16.8761 & 0.7126 & 0.2182 & -0.2647\\ \cline{2-6}
 %& 0.0125 & 24.3951 & 0.5316 & 0.1394 & -0.6467 \\ \hline \hline
& 0.0125 & 24.3951 & 0.4866 & 0.1394 & -0.5043 \\ \hline \hline
\multirow{2}{*}
{$n=h^{-2}$} 
 & 0.1000 & 2.8101 & & 0.1348 &\\ \cline{2-6}
 %& 0.0500 & 2.5954 & -0.1147 & 0.0557 & -1.2754\\ \cline{2-6}
 %& 0.0250 & 2.5372 & -0.0327 & 0.0349 & -0.6753\\ \cline{2-6}
 %& 0.0125 & 2.7125 & 0.0964  & 0.0167 & -1.0603 \\ \hline
 & 0.0125 & 2.7125 & -0.0170  & 0.0167 & -1.0037 \\ \hline \hline
 \multirow{2}{*}{$n=h^{-3}$} 
 & 0.1000 & 0.9637 & & 0.0537 &\\ \cline{2-6}
 %& 0.0500 & 0.6439 & -0.5816 & 0.0142 & -1.9198 \\ \cline{2-6}
 %& 0.0250 & 0.4928 & -0.3860 & 0.0057 & -1.3182 \\ \cline{2-6}
 %& 0.0125 & 0.3094 & -0.6716 & 0.0017 & -1.7568 \\ \hline
 & 0.0125 & 0.3094 & -0.5464 & 0.0017 & -1.6649 \\ \hline \hline
 \multirow{2}{*}
{$n=h^{-4}$} 
 & 0.1000 & 0.6325 & & 0.0380 &\\ \cline{2-6}
 %& 0.0500 & 0.3266 & -0.9534 & 0.0097 & -1.9676 \\ \cline{2-6}
 %& 0.0250 & 0.1661 & -0.9760 & 0.0023 & -2.0790 \\ \cline{2-6}
 %& 0.0125 & 0.0838 & -0.9868 & 6.3816e-4 & -1.8502 \\ \hline
 & 0.0125 & 0.0838 & -0.9721 & 6.3816e-4 & -1.9656 \\ \hline
\end{tabular} \vskip0.2cm
\caption{The convergence rate $\alpha$ in the $H^1$ norm and $\beta$ in the $L^2$ norm on the unit square.}\label{tab.1}
\end{table}
\end{center}

\begin{example}\label{numerical.2}
We take $\Omega$ be a unit circle. The mesh is depicted in Figure \ref{example.2.1}. We set $\{x_i\}_{i=1}^n$ being uniformly distributed on $\Ga$, and let $e_i=\eta_i+\alpha_i,~ i=1,2,\cdots, n$, where $\eta_i$ and $\alpha_i$ are independent normal random variables with variance $\sigma_1=1$ and $\sigma_2=10e_i,~ i=1,2,\cdots, n$. We take different $n=h^{-i}$, $i=1,2,3,4$. Figure \ref{example.2} shows the convergence rate of the error in the $H^1$ and $L^2$ norm 
for each choice of $n$. Table \ref{tab.2} show the convergence rate $\alpha$ in the $H^1$ norm and the convergence rate $\beta$ in the $L^2$ norm. Here again we observe the numerical results confirm our theoretical analysis.
\end{example}

\begin{figure}
\includegraphics[width=5cm]{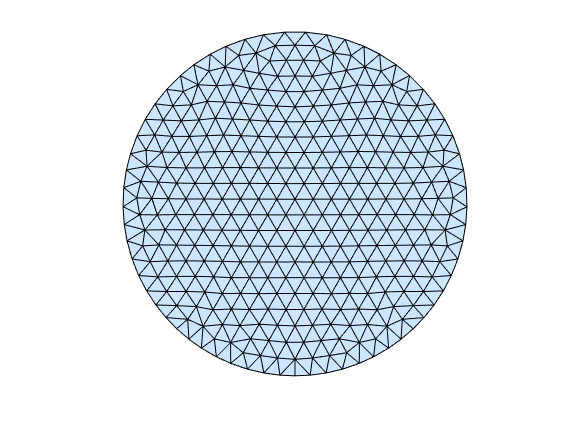}
\caption{The uniform mesh for the unit circle with mesh size $h=0.1$.}
\label{example.2.1}
\end{figure}

\begin{figure}
\subfigure[$H^1$ convergence]{\includegraphics[width=6cm]{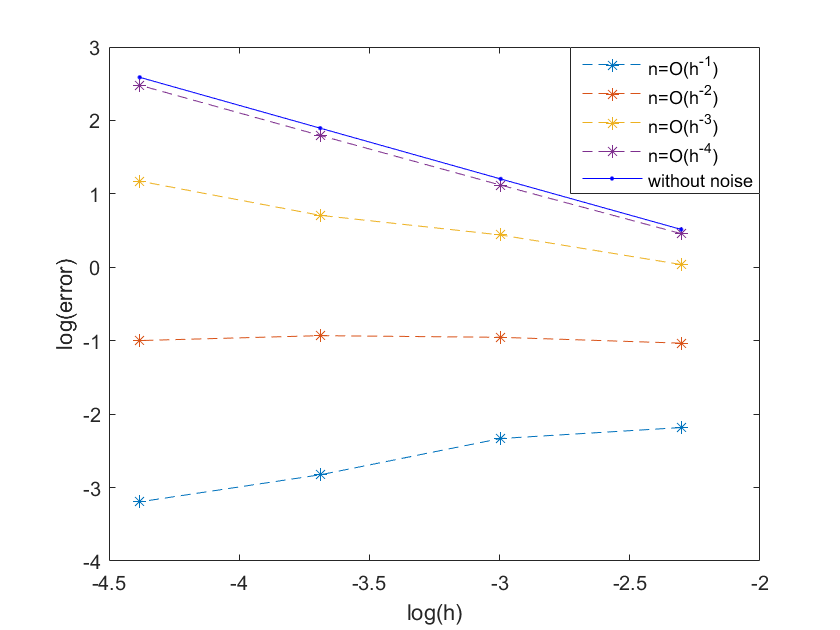}}
\subfigure[$L^2$ convergence]{\includegraphics[width=6cm]{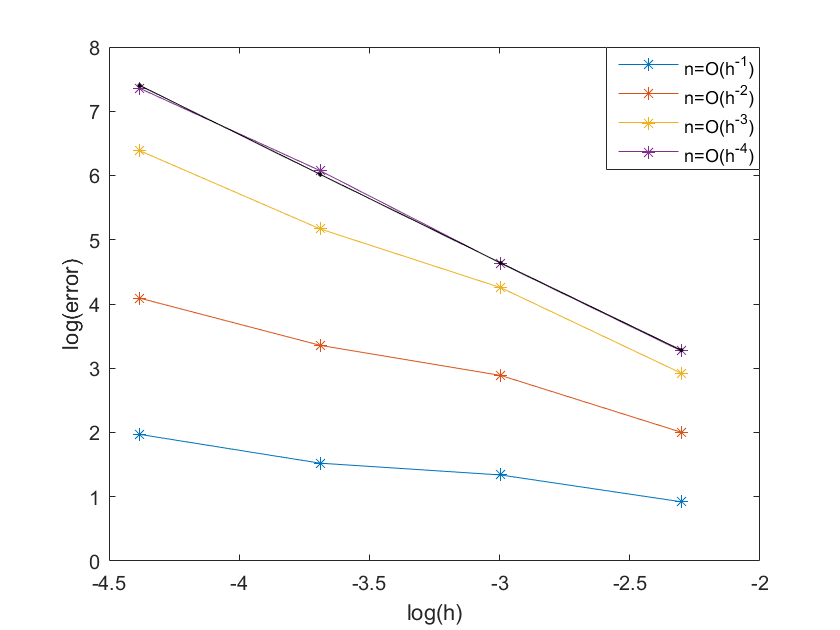}}
\caption{The log-log plot of the convergence rate on the unit circle.}
\label{example.2}
\end{figure}

\begin{center}
\begin{table}
\begin{tabular}{ | c | c | c | c | c | c |}
\hline
$n$ & $h$ & $H^1$ error & $\alpha$ & $L^2$ error & $\beta$ \\ \hline \hline
\multirow{2}{*}
{$n=h^{-1}$} 
 & 0.1000 & 46.5037   &             & 2.5872 &\\ \cline{2-6}
& 0.0125 & 127.832 & 0.4863 & 0.9527 & -0.4804 \\ \hline \hline
\multirow{2}{*}
{$n=h^{-2}$} 
 & 0.1000 & 13.1775 & & 0.8668 &\\ \cline{2-6}
 & 0.0125 & 16.1040 & 0.0964  & 0.1133 & -0.9787 \\ \hline \hline
 \multirow{2}{*}{$n=h^{-3}$} 
 & 0.1000 & 5.4157 & & 0.2924 &\\ \cline{2-6}
 & 0.0125 & 1.7581 & -0.5410 & 0.0113 & -1.5665 \\ \hline \hline
 \multirow{2}{*}
{$n=h^{-4}$} 
 & 0.1000 & 2.0009 & & 0.0980 &\\ \cline{2-6}
 & 0.0125 & 0.2527 & -0.9950 & 0.0016 & -1.9790 \\ \hline
\end{tabular} \vskip0.2cm
\caption{The convergence rate $\alpha$ in the $H^1$ norm and $\beta$ in the $L^2$ norm on the unit circle.}\label{tab.2}
\end{table}
\end{center}

%%%%%%%%%%%%%%%%%%%%%%%%%%%%%%%%%%%%%%%%%%%%%%%%%%%%%%%%%%%%%%

\end{document}